\documentclass[12pt,reqno]{amsart}
\usepackage{amsmath}
\usepackage{amssymb}
\usepackage{amsfonts, stmaryrd}
\usepackage{a4wide}
\newcommand{\CC}{\mathbb{C}}

\newcommand{\TT}{\mathbb{T}}
\newcommand{\DD}{\mathbb{D}}

 \DeclareMathOperator{\dd}{d}
\newtheorem{thm}{{Theorem }}
\newtheorem{lem}{{Lemma }}

\newcommand{\cA}{{\mathcal{A}}}
\newcommand{\cB}{{\mathcal{B}}}

\begin{document}
\author[El-Fallah]{Omar El-Fallah}
\address{D\'{e}partement de Math\'ematiques
\\Universit\'e Mohamed V\\B.\ P.\ 1014 Rabat\\Morocco}
\email{elfallah@fsr.ac.ma}
\author[Kellay]{Karim Kellay}
\address{CMI\\LATP\\Universit\'e de Provence\\
39, rue F. Joliot-Curie\\13453 Marseille Cedex 13\\France}
\email{kellay@cmi.univ-mrs.fr}
\author[Seip]{Kristian Seip}
\address{Department of Mathematical Sciences\\ Norwegian University of Science and Technology\\
NO-7491 Trondheim\\ Norway} \email{seip@math.ntnu.no}
\thanks{The first  author was supported by Egide-Volubilis. The second author was supported by Egide-Volubilis and ANRDynop.
The third author was supported by the Research Council of Norway
grant 185359/V30. This work started when the third author was
visiting CMI, and he thanks LATP for its support and
hospitality}.
\title[Cyclicity of singular inner functions]{Cyclicity of singular inner functions \\
from the  corona theorem}
\date{\today}


\maketitle

\begin{abstract}
Carleson's corona theorem is used to obtain two results on cyclicity
of singular inner functions in weighted Bergman-type spaces on the
unit disk. Our method proof requires no regularity conditions on the
weights.
\end{abstract}

\section{introduction}
This paper studies cyclicity of singular inner functions in two
different classes of weighted Bergman-type spaces. In both cases,
our proofs rely crucially on Carleson's corona theorem. An
interesting feature of this method of proof is that regularity
conditions on the weights can be avoided.

We begin by considering weighted $\ell^2$ spaces, viewed as spaces
of analytic functions in the unit disk. We say that a sequence of
positive numbers $\omega=(\omega(n))_{n\geq 0}$ is a weight sequence
if $\omega(n)\nearrow +\infty$ and $\log \omega(n)=o(n)$ when
$n\to\infty$. With every weight sequence $\omega$ we associate the
weighted Bergman space $\cA_\omega^2$ which consists of all analytic
functions $f(z)=\sum_{n=0}^\infty a_n z^n$ on the open unit disk
$\DD$ such that
$$\|f\|_{\omega,2}^2 =\sum_{n\geq 0} \frac{|a_n|^2}{\omega(n)^2} <\infty.$$
A function $f$ in $\cA^2_\omega$ is said to be cyclic in
$\cA^2_\omega$ if the set of functions $pf$ with $p$ a polynomial is
dense in $\cA^2_\omega$.


We will prove the following theorem.

\begin{thm}\label{cyclic} Let $\omega$ be a weight sequence. If
\begin{equation}
\label{newquasi}\sum_{n\ge 1}\frac{(\log
\omega(n))^2}{n^{2}}=\infty,
\end{equation}
then every function in $H^\infty$ without zeros in $\DD$ is cyclic
in $\cA^{2}_{\omega}$;

\end{thm}

This result has its roots in work of Beurling \cite{B} and Nikolskii
\cite[\S 2.6, Theorem 2]{N}. Requiring a certain regularity
condition on $\omega$, Beurling proved that every function in
$H^\infty$ without zeros in $\DD$ is cyclic in $\displaystyle
\bigcup_{n\geq 1}\cA^{2}_{\omega^n}$, equipped with the inductive
limit topology,  if and only
\begin{equation}
\label{quasicond} \sum_{n\ge 1}\frac{\log
\omega(n)}{n^{3/2}}=\infty.
\end{equation}
 The Hilbert space case was considered by Nikolskii \cite[\S
2.6, Theorem 2]{N} who proved that whenever $\omega$ is log-concave,
i.e., $\omega^2(n)\geq \omega(n+1)\omega(n-1)$), the divergence
condition \eqref{quasicond} implies that every function in
$H^\infty$ without zeros in $\DD$ is cyclic in $\cA^2_\omega$.
Beurling used Bernstein's approximation theorem, while Nikolskii
relied on a theorem on quasi-analyticity that requires the
log-concavity condition. Beurling pointed out at the end of his
paper that he could not dispense with a certain convexity condition
and that it remained an open problem to obtain a general sufficient
condition for cyclicity.

Thus the novelty of Theorem~1, besides the method of proof, is the
absence of any regularity condition on $\omega$. In the second
remark after Theorem 2 below, we will give an example showing that
Theorem~1 enables us to deal with weights that are not covered by
Nikolskii's theorem, in spite of the fact that the divergence
condition \eqref{newquasi} implies \eqref{quasicond}.

We now turn to our second result, which deals with a situation in
which growth restrictions are nonradial. Let $E$ be a closed subset
of $\TT$ and let $\Lambda$ be a nonincreasing and positive function
on $]0,2]$ such that $\Lambda (0^+)=+\infty$. We denote by
$\cB^\infty_{\Lambda,E}$ the space of all analytic functions $f$ on
$\DD$ such that \[ \|f\|_{\Lambda,E,\infty}= \sup_{z\in \DD}
|f(z)|e^{-\Lambda (\dd(z,E))};
\]
here $\dd(\cdot,\cdot)$ stands for Euclidean distance on $\CC$. Let
$I$ be the singular inner function defined by
$$I(z)=e^{-\frac{1+z}{1-z}}.$$
Gevorkyan and Shamoyan  showed in \cite{GS} that if $E=\{1\}$ and
$\Lambda $ satisfies certain regularity conditions, then $I$ is
cyclic in $\cB_{\Lambda,\{1\}}^\infty $ if and only if $\Lambda$
fails to be integrable. We will now prove the same result without
any additional assumption on $\Lambda$.
\begin{thm}
The singular inner function $I$ is cyclic in
$\cB_{\Lambda,\{1\}}^\infty$ if and only if
\begin{equation}\label{Lambdanonint}
\int _0^2\Lambda (t)dt=\infty.
\end{equation}
\end{thm}

Several remarks are in order before we turn to the proofs of our
theorems: \vspace{3mm}

\textbf{1.} Cyclicity of the singular inner function $I$ was first
considered in weighted approximation theory by Keldy\v{s} \cite{K}
and Beurling \cite{B}. See also \cite[\S 2.8, Theorem 1]{N}. The
idea of using the corona theorem in this context goes back to
Roberts \cite{R}. \vspace{3mm}

\textbf{2.} We next give an example of a weight that satisfies
\eqref{newquasi} but to which Nikolskii's theorem does not apply.
Set
\[ \log \omega(n)= 2^{2^{j-1}}, \ \ \ n\in
[2^{2^j},2^{2^{j+1}}) \] for $j=1,2, ...$. Then clearly
\eqref{newquasi} holds, but we may check that if
$\widetilde{\omega}$ is log-concave and $\widetilde{\omega}\le
\omega$, then
\[
\sum_{n\ge 1}\frac{\log \widetilde{\omega}(n)}{n^{3/2}}<+\infty.
\]
To see this, let $g$ be the linear function that satisfies
$g(2^{2^{j-1}})=2^{2^{j-3}}$ and $g(2^{2^j})=2^{2^{j-2}}$. By
concavity, $\widetilde{\omega}(n)\le \min(2^{2^{j-1}}, g(n))$ for
$n$ in the interval $[2^{2^j},2^{2^{j+1}})$. A straightforward
computation shows that the piecewise linear function
\[ h(n)=\min(2^{2^{j-1}}, g(n)), \ \ \ n\in [2^{2^j},2^{2^{j+1}}) \]
satisfies
\[ \sum_{n>4} \frac{h(n)}{n^{3/2}}<\infty. \] \vspace{3mm}

\textbf{3.} The assumption that $\omega$ is nondecreasing implies
that the shift operator is a contraction on $\cA_{\omega}^2$, a fact
that plays an essential role in the proof of Theorem 1 given below.
One may ask if this monotonicity can be dispensed with. While we can
not rule out the possibility that it can be relaxed, the following
example shows that it can not simply be removed. Namely, let
$\omega$ be any sequence such that $\omega (2n)= 1$. Then $U(z^2)$
is not cyclic in $\cA^2_{_\omega}$ when $U$ is an arbitrary inner
function. Indeed, if there is a sequence of polynomials $p_n$ such
that $\|p_n (z)U(z^2)-1\|_{\omega,2}\to 0$, we may write $p_n(z)=
q_n(z^2)+zh_n(z^2)$; then $\|q_n (z^2)U(z^2)-1\|_{\omega,2}\to 0$
which means that $U$ is cyclic in $H^2$. But this contradicts
Beurling's theorem.  \vspace{3mm}

\textbf{4.} We have the following relation between the two kinds of
Bergman spaces considered in the present work. We denote by $\mu$
normalized Lebesgue measure on $\DD$ and define
$\cB_{\Lambda,\TT}^{2}$ to be the space of all analytic functions
$f$ on $\DD$ such that
\[ \|f\|_{\Lambda,\TT,2}^2=\int_\DD
|f(z)|^2e^{-2\Lambda(1-|z|)}d\mu(z)<\infty.
\]
A simple computation shows that $f(z)=\sum_{n=0}^\infty a_n z^n$
belongs to $\cB_{\Lambda,\TT}^{2}$ if and only if
$$
\|f\|_{\Lambda,\TT,2}^{2}=\sum_{n\geq 0}\frac{|a_n|^2}{ \omega(n)^2}<
\infty,
$$
where
\begin{equation}\label{moment}
\omega(n)^{-2}=\int_{0}^{1}r^{2n+1}e^{-2\Lambda(1-|z|)} dr,\qquad
n\geq 0.
\end{equation}
Note that this moment sequence $\omega$ is log-concave. Conversely,
if $\omega$ is a log-concave weight sequence, then there exists a
$\Lambda$ such that
\[ c\omega(n)\le \left(\int_{0}^{1}r^{2n+1}e^{-2\Lambda(1-|z|)} dr\right)^{-1/2}\le C \omega(n)\]
for positive constants $c$ and $C$ independent of $n\ge 1$
 \cite[Proposition 4.1]{BH}.
 However, in general, we can not write
 $\cA^2_\omega$ as $\cB_{\Lambda,\TT}^2$.
\vspace{3mm}

\textbf{5.} In \cite[\S 2.6, \S 2.7]{N}, Nikolskii proved, under
some regularity conditions on $\Lambda$, that $I$ is cyclic in
$\cB_{\Lambda,\TT}^2$ if and only if
\begin{equation}\label{nikolski condition}
\displaystyle \int _0^2\sqrt{\frac{\Lambda (t)}{t}}dt =\infty.
\end{equation}
It is interesting to note when $\Lambda$ and $\omega$ are related as
in \eqref{moment}, then, under suitable regularity
 conditions on $\Lambda$,
\eqref{nikolski condition} is equivalent to
 \eqref{quasicond} and \eqref{Lambdanonint} is equivalent to \eqref{newquasi}
 \cite[\S 2.6 Lemma~1, Lemma~2]{N}. As will be pointed out in Section 3 below,
 a slight variant of our proof of Theorem~2 gives
 that \eqref{Lambdanonint} is in fact sufficient for every singular inner function $U$ to be
 cyclic in $\cB_{\Lambda,\TT}^2$. Note that, again, no
 additional regularity condition on $\Lambda$ is required.

 \section{Proof of Theorem 1}

Our proof of Theorem~\ref{cyclic} will rely on three lemmas.

The first lemma gives a convenient reformulation of condition
\eqref{newquasi} of Theorem~\ref{cyclic}. We will use only one of
the implications of the lemma, but we find the result to be  of some
general interest and give therefore the simple proof of the full
equivalence between the two conditions.

\begin{lem}\label{equiv}
Let $\omega$ be a weight sequence. Then the divergence condition
\eqref{newquasi} of Theorem~\ref{cyclic} holds if and only if there
exists a sequence of nonnegative integers $(n_j)_{j\ge 0}$ such that
$\log \omega(n_{j+1})\ge 2\log \omega(n_j)$ and
\begin{equation}
\label{quasi} \sum_{j\ge 0}\frac{(\log\omega(n_j))^2}{{n_j}}=\infty.
\end{equation}
\end{lem}
\begin{proof} If \eqref{newquasi}  holds, then we define $n_j$ inductively by
setting $n_0=1$ and requiring
$$n_{j+1}=\min\{n>n_j:\  \log\omega(n)\ge 2\log\omega(n_j)\}.$$
Since
\[\sum_{n=n_j}^{n_{j+1}-1}\frac{(\log \omega(n))^2}{n^{2}}\le
4 (\log\omega(n_j))^2 \sum_{n=n_j}^{n_{j+1}-1} \frac{1}{n^{2}}\le 4
\frac{(\log\omega(n_j))^2}{{n_j-1}}, \] we conclude that condition
\eqref{newquasi}  of Theorem~\ref{cyclic} implies \eqref{quasi}.

To prove the reverse implication, we observe that if
$\log\omega(n_{j+1})\ge 2 \log\omega(n_j)$ for every $j$, then
\[ \sum_{l \text{ : } \frac{n_j}{2}\le n_l \le n_j}\frac{(\log\omega(n_l))^2}{{n_l}}
\le 4\frac{(\log\omega(n_j))^2}{{n_j}}.\] Thus we may
assume that the sequence $(n_j)$ satisfies the exponential growth
condition $n_{j+1}\ge 2 n_j$, in which case we have
\[ \frac{(\log \omega (n_j))^2}{{n_j}}
\le4\sum_{n=n_j}^{n_{j+1}-1}\frac{(\log \omega(n))^2}{n^2},
\]
and so \eqref{quasi} implies the divergence condition \eqref{newquasi} of
Theorem~\ref{cyclic}.
\end{proof}

Since $\omega$ is nondecreasing, the shift operator $S$ defined by
$Sf(z)=zf(z)$ is a contraction on $\cA^{2}_{\omega}$. It follows
that for $f$ in $H^\infty$, $f(S)$ makes sense by $H^\infty$
functional calculus. Therefore, by von Neumann's inequality, we have
$$\|f\varphi\|_{\omega,2}=\|f(S)\varphi\|_{\omega,2}\leq
\|f(S)\|\|\varphi\|_{\omega,2}\leq\|f\|_\infty\|\varphi\|_{\omega,2}$$
for every $\varphi$ in $\cA^{2}_{\omega}$. The next lemma is a
simple consequence of this fact.

\begin{lem}\label{lemsimple}
Let $U$ be a singular inner function and $\lambda$ a positive
number. Then $U$ is cyclic in $\cA^{2}_{\omega}$ if and only if
$U^\lambda$ is cyclic in $\cA^{2}_{\omega}$.
\end{lem}
\begin{proof}
Since  functions in $H^\infty$  are  multipliers, it suffices to
prove the lemma in the special case when $\lambda=2$.

Suppose first that $U$ is cyclic. Then there exist  polynomials
$p_n$ such that $\|1-p_nU\|_{\omega,2} \to 0$ as $n\to \infty$. So
$\|U-p_nU^2\|_{\omega,2}\leq \|U\|_\infty\|1-p_nU\|_{\omega,2}\to 0$
and $U^2$ is cyclic.

We now assume that $U^2$ is cyclic. Then there exist polynomials
$p_n$ such that $\displaystyle \|1-p_nU^2\|_{\omega,2}\to 0$ when
$n\to \infty$. So the functions $f_n=p_nU$ in $H^\infty$ have the
property that $\|1-f_nU\|_{\omega,2}\to 0$, which means that $U$ is
cyclic since $\omega(n)\to \infty$, we have
\[ \overline{\{pV \text{
: } p \text{ polynomial} \}}^{\cA^{2}_{\omega}}= \overline{\{ f V
\text{ : }f\in H^\infty\}}^{\cA^{2}_{\omega}}\] for every function
$V$ in $H^\infty$.
\end{proof}

We now turn to our application of the corona theorem.
\begin{lem}\label{corona} Let $\nu$ be the singular measure of $U$ and set $c^2=\nu(\TT)$. Then
for every nonnegative integer $n$, 
there is
a function $f_n$ in $H^\infty$ such that
\begin{eqnarray}
\label{norm1} \|1-f_nU\|_{\omega,2} & \leq &
\frac{e^{A(c\sqrt{n}+1)}}{\omega(n)}, \\
\label{norm2} \|f_n\|_{\infty}&\leq& e^{A({c}\sqrt{n}+1)},
\end{eqnarray}
with $A$ an absolute constant.
\end{lem}
\begin{proof}
Note that
\begin{eqnarray*}
\inf_{z\in \DD}\Big[|U(z)| +|z|^n\Big]&\geq&
\inf_{z\in \DD}\Big[ \expÊ\frac{-2c^2}{1-|z|} + |z|^n\Big]\\
&\geq& e^{-2c\sqrt{n}}=\delta_n.
\end{eqnarray*}
By Carleson's corona theorem \cite{Ca}, \cite[p. 66]{Nik}, there
exist $f_n,g_n\in H^\infty$ such that
$$\left\{
\begin{array}{lll}
f_n U+ z^n g_n =1,\\
\\
\|f_n\|_\infty\leq 2^5\delta_n^{-3},\ \ \ \|g_n\|_\infty\leq 2^5\delta_n^{-3},\\
\end{array}
\right.
$$
which implies that \eqref{norm2} is satisfied for some absolute
positive constant $A$. 
We observe that
\eqref{norm1} also holds by the observation above and the estimate
$$\|1-f_nU\|_{\omega,2}=\|z^n g_n\|_{\omega,2}\leq \|g_n\|_\infty \|z^n\|_{\omega,2}.$$
\end{proof}

\subsection*{Proof of Theorem~\ref{cyclic}}
 Let $f$ be a function in
$ H^\infty$ without any zeros in $\DD$. We write $f= FU$, where $F$
is an outer function and $U$ is a singular inner function with
associated singular measure $\nu$.  Since $F$ is outer, $F$ is
cyclic in $H^2\subset \cA^2_\omega$. Hence it remains only to prove
that $U$ is cyclic in $\cA^2_\omega$.


Let $m_1,m_2,\ldots,m_N$ be arbitrary positive integers and
$\lambda_j$ associated positive numbers such that
$$\sum_{j=1}^{N}\lambda_j^2=1.$$
Set $U_j=U^{\lambda_j^2}$ and let $f_j$ be a function such that
$$
\left\{\begin{array}{ll}
\displaystyle \|1-f_jU_j\|_{\omega,2}\leq \exp[A(c\lambda_j\sqrt{m_j}+1)-\log \omega(m_j)],\\
\\
\displaystyle \|f_j\|_{\infty}\leq e^{A(c\lambda_j\sqrt{m_j}+1)}, \; \;\;  j=1,\ldots,N.\\
\end{array}
\right.
$$
By Lemma \ref{corona}, such a function exists for every $j$, with
$A$ an absolute constant. Since
$$1-U\prod_{j=1}^{N}f_j=1-U_1f_1+ U_1f_1(1-U_2f_2)+\ldots+ \prod_{j=1}^{N-1}U_jf_j(1-U_Nf_N),$$
we get
$$\|1-U\prod_{j=1}^{N}f_j\|_{\omega,2}\leq\sum_{j=1}^{N}\exp\Big[\sum_{k=1}^{j}
A(c\lambda_k\sqrt{m_k}+1)- \log\omega(m_j)\Big].$$ Now choose
$m_j=n_{j+j_0}$ for some $j_0$, where $(n_j)$ is the sequence
obtained from Lemma~\ref{equiv}. This means that $\log
\omega(n_{j+1})\ge 2\log \omega(n_j)$ and also that \eqref{quasi}
holds. Let $N=N(j_0)$ be such that
$$N=\min\left\{M\text{ : } \sum_{j=1}^{M}\frac{(\log\omega(n_{j+j_0}))^2}{{n_{j+j_0}}}\geq (4Ac)^2\right\}$$
and set
$$\lambda_j=\frac{\log\omega(n_{j_0+j})}{\sqrt{n_{j_0+j}}}
\Big(\sum_{k=1}^{N}\frac{(\log\omega(n_{j_0+k}))^2}{{n_{j_0+k}}}\Big)^{-1/2}.
$$
By our choice of sequence $(n_j)$, we have then
$$\lambda_j\leq \frac{1}{4Ac}\frac{\log\omega(n_{j_0+j})}{\sqrt{n_{j_0+j}}}
\ \ \ \text{and} \ \ \  \sum_{k=1}^{j}\log\omega(n_{j_0+k})\leq
2\log \omega(n_{j_0+j}).$$ Thus we get
\[
\|1-U\prod_{j=1}^{N}f_j\|_{\omega,2}\leq \sum_{j=1}^{N}e^{Aj-\frac{1}{2}\log\omega(n_{j_0+j})}\\
 \leq \frac{C}{\sqrt{\omega(n_{j_0+1})}}
\]
for an absolute constant $C$. This finishes the proof since
$\omega(n_{j_0+1})\to \infty$ when $j_0\to \infty$.

\section{Proof of Theorem 2}

For the proof of Theorem~2, we need the following two lemmas.

\begin{lem}\label{function-delta}
Suppose that $0<\delta<1$ and let $f_{\delta }$ be the outer
function defined by
$$
f_{\delta}(z)=\exp \Big({-\Lambda (\delta)\displaystyle \int _{\delta<|t|<\pi}\frac{e^{it}+z}{e^{it}-z}dt }\Big).
$$
If $a$ is in $]0,(2\pi)^{-1}]$, then we have \begin{eqnarray}
|f_{\delta}(z)|& \geq & e^{-4\pi^2\Lambda (\delta)a} \ \ \text{for}
\ \
  \frac{|1-z|^2}{1-|z|^2}\leq a\delta , \label{first}  \\
  \|f_{\delta}\|_{\Lambda,\{1\},\infty}& \leq  & e^{-\pi \Lambda (\delta)}.
\label{second}
\end{eqnarray}
\end{lem}
\begin{proof}
We first prove \eqref{first}. Let $z$ be a point in  $\DD$ such that
$ \frac{|1-z|^2}{1-|z|^2}\leq a \delta$. Then we have $1-|z|\le
|1-z|\le 2 a\delta$, which implies that
\[
\log|f_{\delta}(z)|^{-1} =  \Lambda (\delta)\displaystyle \int
_{\delta < |t|< \pi}\frac{1-|z|^2}{|e^{it}-z|^2}dt \le
8 a\delta \Lambda(\delta) \int _{\delta < t < \pi} \frac{dt}{(|e^{it}-1|-2a\delta)^2}.
\]
Using that $|e^{it}-1|\ge 2t/\pi$ for $0\le t\le \pi$ and that
$a\le(2\pi)^{-1}$, we obtain the desired estimate \eqref{first}.

We will now prove \eqref{second}; we will do this by showing that
for every $z$ in $\DD$, we have
\begin{equation}\label{tobeproved} |f_{\delta }(z)|\leq e^{-\pi\Lambda
(\delta)+\Lambda (|1-z|)}.\end{equation} When $|1-z|\le \delta$,
this inequality holds because $|f_{\delta}(z)|\leq 1$ and $\Lambda$
is a decreasing function. To deal with the case $|1-z|>\delta$, we
argue as follows. Let $E_\delta$ be the sub-arc of points $e^{it}$
on the unit circle satisfying  $\delta \le |t| \le \pi$. Then we may
write
\[
 \log |f_{\delta}(z)| = -\Lambda (\delta)
\displaystyle \int _{ \delta<|t|<
\pi}\frac{1-|z|^2}{|e^{it}-z|^2}dt=- \Lambda(\delta)2\pi \varpi(z,
E_\delta, \DD),\] where $\varpi(z, E_\delta, \DD)$ denotes harmonic
measure of $E_\delta$ at $z$ in $\DD$. A simple geometric argument
shows that when $|z-1|>\delta$, $z$ lies in the domain bounded by
$E_\delta$ and the hyperbolic geodesic between the endpoints of
$E_\delta$. Therefore, $\varpi(z, E_\delta, \DD)\ge 1/2$, and
\eqref{tobeproved} follows.
\end{proof}

\begin{lem}\label{corona2}
Let $c$ be a positive number and $n$ a positive integer such that
$4\pi^2cn\le \Lambda (1/n)$, and set $I_c(z) =
e^{-c\frac{1+z}{1-z}}$. Then there exists a bounded analytic
function $g_n$ such that
\[ \aligned
 \|1-g_nI_c\|_{\Lambda,\{1\},\infty} & \leq
e^{B\sqrt{cn\Lambda (1/n)}-\pi \Lambda (1/n)}, \\
 \|g_n\|_{\infty} &\leq e^{B\sqrt{cn\Lambda
  (1/n)}},
\endaligned \]
where $B$ is a universal constant.
\end{lem}
\begin{proof}
Applying Lemma~\ref{function-delta} with $\delta =\delta_n=1/n$ and
$a= \sqrt{\frac{cn}{\Lambda (1/n)}}$, we obtain
\[
|f_{\delta_n}(z)|+|I_c(z)|\geq \min (e^{-4\pi^2a\Lambda
(\delta)},e^{-\frac{c}{a\delta}})=  e^{-4\pi^2\sqrt{cn\Lambda
(1/n)}},
\]
and
\[
\|f_{\delta_n}\|_{\Lambda,\{1\},\infty}\leq  e^{-\pi \Lambda (1/n)}.
\]
By the corona theorem, we obtain the desired estimates.
\end{proof}

\subsection*{Proof of Theorem~2}
Assume first that $\Lambda$ is integrable. We will use Keldy\v{s}'s
method \cite{K,N} to prove that this implies that $I$ is not cyclic.
So suppose to the contrary that $I$ is cyclic. Then there exists a
sequence of polynomials $(p_n)$ such that
$\|Ip_n-1\|_{\Lambda,\{1\},\infty}\to 0$. Thus, in particular, we
have $C=\sup _n\|Ip_n\|_{\Lambda,\{1\},\infty} <\infty$. Since
$|I^*(\zeta)| =1$ for all $\zeta$ in $\TT \setminus \{1\}$, we
obtain $|p_n(\zeta)|\leq Ce^{-\Lambda (|1-\zeta|)}$. Let $F$ be the
outer function given by
$$
F(z)=\exp \Big({-\displaystyle \int _0^{2\pi}\frac{e^{it}+z}{e^{it}-z}\Lambda (|1-e^{it}|)\frac{dt }{2\pi}}\Big).
 $$
By the  generalized maximum principal, $|p_n(z)|\leq C |F(z)|$ for
all $|z|<1$. But we also have $\displaystyle \lim _{n\to \infty }
p_n(z)= 1/|I(z)|$, so that $|I(z)|^{-1} \leq |F(z)|$, but this is
impossible since $I$ is a singular inner function and $F$ is an
outer function.

The proof of the converse  is essentially the same as the proof of
Theorem \ref{cyclic}, and we will therefore only sketch the
argument. We begin by noting that
$$\displaystyle \int _0^2 \Lambda (t)dt =\infty \ \ \ \Leftrightarrow \ \ \ \displaystyle
\sum _{n\geq 1}\frac{\Lambda (1/n)}{n^2}=\infty.$$ So, by Lemma
\ref{equiv}, there exists a sequence $(n_j)$ such that
\[
\Lambda (n_{j+1})\geq 2\Lambda (n_j)\quad \text{and}\quad  \sum _{n\geq 1}\frac{\Lambda (1/n_j)}{n_j}=\infty.
\]
We fix $j_0$ and choose $N$ so large that
\begin{equation}\label{assumeonN} \sum_{j=j_0+1}^N
\frac{\Lambda(1/n_j)}{n_j}\ge 4\pi^2 B^2.
\end{equation}
Following the scheme of proof for Theorem 1, we  make the choice
\[ U_j=I_{\lambda_j^2},\]
where
 \[
\lambda_j^2= \frac{\Lambda (1/n_j)}{n_j}\left(\displaystyle
\sum_{k=1}^N\frac{\Lambda (1/n_{j_0+k})}{n_{j_0+k}}\right) ^{-1}.
 \]
By our assumption \eqref{assumeonN}, Lemma~\ref{corona2} applies
with $c=\lambda_j^2$ and $n=n_j$ for $j_0+1\le j\le j_0+N$. The rest
of the proof follows step by step the last part of the proof of
Theorem~1. We omit the details.

Let us  note that if we in the latter argument replace the
function $f_{\delta_n}$ in Lemma~\ref{corona2} by
$z^{[n\Lambda(1/n)]}$ (here $[x]$ denotes the integer part of $x$),
then we obtain the result mentioned in Remark 5 of the introduction:
If $\Lambda$ fails to be integrable, then every singular inner
function is cyclic in $\cB_{\Lambda,\TT}^\infty$.

We finally mention that, by the same method of proof, we may replace
$\cB_{\Lambda,\{1\}}^{\infty}$ in Theorem~2 by
$\cB_{\Lambda,\{1\}}^{2}$, which is the Hilbert space of analytic
functions $f$ on $\DD$ such that
 \[
 \|f\|_{\Lambda,\{1\},2}^2=\int_\DD
|f(z)|^2e^{-2\Lambda(|1-z|)}d\mu(z)<\infty.
\]

\end{document}